\documentclass{svproc}
\usepackage{url}
\usepackage{amssymb,amsfonts,amsmath}
\usepackage[labelformat=empty]{caption}
\usepackage[labelformat=empty]{caption}
\usepackage{tcolorbox}
\usepackage{tikz-cd}
\usepackage{aliascnt}
\usepackage{hyperref}

\usepackage{enumitem}
\setlist[enumerate]{label=$(\mathrm{\arabic*})$, leftmargin=*}
\setlist[itemize]{leftmargin=*}

\newtheorem{thm}{Theorem}

\newaliascnt{theo}{thm}
\newtheorem{theo}[theo]{Theorem}
\aliascntresetthe{theo}

\newaliascnt{cor}{thm}
\newtheorem{cor}[cor]{Corollary}
\aliascntresetthe{cor}

\newaliascnt{prop}{thm}
\newtheorem{prop}[prop]{Proposition}
\aliascntresetthe{prop}

\newaliascnt{lem}{thm}
\newtheorem{lem}[lem]{Lemma}
\aliascntresetthe{lem}

\newaliascnt{conj}{thm}

\aliascntresetthe{conj}

\newaliascnt{que}{thm}

\aliascntresetthe{que}

\newaliascnt{ass}{thm}

\aliascntresetthe{ass}

\newaliascnt{defnot}{thm}

\aliascntresetthe{defnot}

\newaliascnt{rem}{thm}

\aliascntresetthe{rem}

\newaliascnt{exmp}{thm}

\aliascntresetthe{exmp}

\newaliascnt{notn}{thm}

\aliascntresetthe{notn}

\newtheorem{conv}[thm]{Convention}

\newcommand{\Br}{{\rm Br}}

\newcommand{\alb}{{\rm alb}}
\newcommand{\CH}{{\rm CH}}

\newcommand{\Hom}{{\rm Hom}}

\newcommand{\Spec}{{\rm Spec \,}}

\newcommand{\Alb}{\rm Alb}

\newcommand{\Gal}{{\rm Gal}}

\newcommand{\NS}{{\operatorname{NS}}}

\newcommand{\et}{{\text{\'et}}}
\newcommand{\tor}{{\operatorname{\rm tor}}}

\newcommand{\ds}{{/\kern-3pt/}}

\makeatletter
\let\c@equation\c@thm
\makeatother
\begin{document}
\mainmatter              
\title{A note on zero-cycles on bielliptic surfaces}
\titlerunning{Zero-cycles on bielliptic surfaces}  
%
\author{Evangelia Gazaki\inst{1}}
\authorrunning{Evangelia Gazaki} 
%
%
\institute{University of Virginia, Charlottesville VA 22904, USA,\\
\email{eg4va@virginia.edu},\\ website:
\texttt{https://sites.google.com/view/valiagazakihomepage/home}}

\maketitle              

\begin{abstract}
We study the Chow group of zero-cycles $\CH_0(S)$ of a bielliptic surface $S=(E_1\times E_2)/G$, where $E_1, E_2$ are elliptic curves and $G$ is a finite group acting on $E_1$ by translations and on $E_2$ by automorphisms such that $E_2/G\simeq\mathbb{P}^1$. We show that if $S$ is defined over an arbitrary field $k$ of characteristic not equal to $2,3$, then the kernel of the Albanese map $\alb_S:\CH_0(S)^{\deg=0}\rightarrow \Alb_S(k)$ is a torsion group of exponent $2^2\cdot|G|$ or $3^2\cdot|G|$, depending on the type of bielliptic surface. We also construct  explicit examples over $p$-adic fields that illustrate that this kernel  can have nontrivial elements obtained by push-forward from the abelian surface. 

\end{abstract}


\section{Introduction} 
For a smooth projective geometrically connected variety $X$ over a field $k$ the Chow group $\CH_0(X)$ of zero-cycles on $X$ admits a degree map 
\[\deg:\CH_0(X)\to\mathbb{Z},\] sending the class $[x]$ of a closed point to the degree, $[k(x):k]$, of the residue field. We will denote by $A_0(X)$ its kernel, which coincides with the subgroup of algebraically trivial cycles. Moreover, there is an Albanese map
\[\alb_X: A_0(X)\to\Alb_X(k),\] where $\Alb_X$ is the Albanese variety associated to $X$. We will denote its kernel by $T(X)$. 

In this short note we study the Chow group of zero-cycles $\CH_0(S)$ of a bielliptic surface $S$ over a field $k$. By a \textit{bielliptic surface} we will always mean a surface $S$ defined as the \'{e}tale quotient $S=(E_1\times E_2)/G$ of a product $X=E_1\times E_2$ of two elliptic curves, where $G$ is a finite abelian group acting on $E_1$ by translations by a $k$-rational nonzero torsion point, and on $E_2$ by automorpisms such that $E_2/G\simeq \mathbb{P}_k^1$.

 Such surfaces have Kodaira dimension $\kappa=0$, geometric genus $p_g=0$ and irregularity $q=1$. Given the first two properties, it follows by the classical work of Bloch, Kas and Lieberman (\cite{BKL76}) that when $k=\overline{k}$ is an algebraically closed field, then $T(S)=0$, which implies that if $S$ is defined over an arbitrary field $k$, then $T(S)$ is a torsion group.
 
When $k$ is a finite field, the unramified class field theory of Kato and Saito (\cite[Proposition 9]{KS83}) gives a description of the Albanese kernel $T(S)$ in terms of the torsion subgroup of the N\'{e}ron-Severi group of $S$, and hence it is generally nontrivial (see Remark \ref{finitefield}).  When $k$ is either an algebraic number field or a $p$-adic field, then it follows by work of Colliot-Th\'{e}l\`{e}ne and Raskind (\cite[Th\'{e}or\`{e}mes A, D, 5.1]{CTR91}) that  $T(S)$ is finite. In general though, to the author's knowledge not much else is known about its structure.  
 
The first theorem of this note gives more explicit information on what torsion may appear in $T(S)$ over an arbitrary base field $k$.  
 
 \begin{theo}\label{mainthmintro} Let $S=(E_1\times E_2)/G$ be a bielliptic surface over a field $k$ of characteristic not equal to $2$ or $3$. Then the Albanese kernel $T(S)$ is a torsion group of exponent $2^2\cdot|G|$ if $2$ divides $|G|$, and $3^2\cdot|G|$ otherwise. 
 \end{theo}
 
 We note that the bielliptic surfaces have been classified into 7 different types (see \autoref{sec:bielliptic}), with the order $|G|$ ranging from $2$ to $9$. To prove \autoref{mainthmintro} we essentially consider each type separately, but using  an intermediate \'{e}tale cover $E_1\times E_2\rightarrow\tilde{S}\rightarrow S$ from another bielliptic surface $\tilde{S}$, we reduce the theorem to bielliptic surfaces of two specific types, corresponding to groups $G=\mathbb{Z}/2\mathbb{Z}$ and $G=\mathbb{Z}/3\mathbb{Z}$. 
 
Our second theorem gives some explicit examples of bielliptic surfaces over $p$-adic fields for which the group $T(S)$ is nontrivial.
 
 \begin{theo}\label{thm2intro} There exist bielliptic surfaces $S$ over $p$-adic fields such that the Albanese kernel $T(S)$ has nontrivial $2$-torsion. 
 \end{theo}
 To emphasize the difference in strategy with the case of finite fields, we note that the nonzero classes in $T(S)$ are obtained by push-forward from the abelian surface, which is chosen to have bad reduction. 
 The key for the construction is the Brauer-Manin pairing, 
 \[\CH_0(S)\times \Br(S)\to\mathbb{Q}/\mathbb{Z},\] which is a bilinear pairing between the Chow group $\CH_0(S)$ and the Brauer group $\Br(S)$ of $S$.



\subsection{Notation} All varieties will be defined over a field $k$ of characteristic not equal to $2$ or $3$. For a variety $X$ over $k$ and an extension $L/k$ we will denote by $X_L$ the base change to $\Spec(L)$. For a closed point $P\in X$ we will denote by $[P]$ its class in $\CH_0(X)$. Moreover, $\Br(X):=H^2_{\et}(X,\mathbb{G}_m)$ will be the cohomological Brauer group of $X$. For an abelian group $B$ and an integer $n\geq 1$ we will denote by $B[n]$ the $n$-torsion subgroup of $B$ and by $B/n$ the $n$-co-torsion.  

\subsection{Acknowledgment} The author's research was partially supported by the  NSF grant DMS-2302196. The author would like to thank Connor Cassady, Roy Joshua and Patrick Brosnan, the organizers of the conference `Arithmetic, $K$-Theory and Algebraic Cycles' for the opportunity to contribute this paper to the conference proceedings. Lastly, the author is grateful to the anonymous referee for useful comments. 

\section{Background} 
In this section we review some necessary background that will be used in the proof of the main results.
\subsection{Bielliptic Surfaces}\label{sec:bielliptic}  We start with some reminders about bielliptic surfaces. Over $\overline{k}$ these have been classified  by Bagnera-De Franchis (\cite{BDF10}) (see also \cite{Suw69}). In what follows we follow the exposition given in the more recent articles \cite[Section 2]{Nue25} and \cite[Section 2]{FTVB22}. To be able to use the classification over an arbitrary base field $k$ (of characteristic not equal to $2, 3$), we make the following convention. 

\begin{conv} In this article by a \textit{bielliptic surface} over $k$ we will always mean a surface $S$ defined as the \'{e}tale quotient $S=(E_1\times E_2)/G$ of a product $X=E_1\times E_2$ of two elliptic curves, where $G$ is a finite abelian group acting on $E_1$ by translations by a $k$-rational nonzero torsion point, and on $E_2$ by automorpisms defined over the base field such that $E_2/G\simeq \mathbb{P}_k^1$.
\end{conv} 

Under this convention, there are $7$ types of such \'{e}tale quotients $S=(E_1\times E_2)/G$.
The different types are presented in Table \ref{Table1}. For each type the Table gives the group $G$, the order of the canonical bundle $K_S$ and the torsion subgroup of $H^2(S_{\overline{k}},\mathbb{Z})$.

 \begingroup
		\renewcommand*{\arraystretch}{1.1}
		\begin{table}\centering
  \begin{tabular}{|c|c|c|c|}
    \hline
    Type & Group & ord$(K_S)$ & $H^2(S_{\overline{k}},\mathbb{Z})_{\tor} $ \\
    \hline
    $1$ & $\mathbb{Z}/2\mathbb{Z}$ & $2$ & $\mathbb{Z}/2\mathbb{Z}\times \mathbb{Z}/2\mathbb{Z}$\\
    \hline
    $2$ & $\mathbb{Z}/2\mathbb{Z}\times\mathbb{Z}/2\mathbb{Z}$  & $2$ & $\mathbb{Z}/2\mathbb{Z}$\\   
    \hline
    $3$ & $\mathbb{Z}/4\mathbb{Z}$ & $4$ & $\mathbb{Z}/2\mathbb{Z}$ \\    
    \hline
    $4$ & $\mathbb{Z}/4\mathbb{Z}\times\mathbb{Z}/2\mathbb{Z}$ & $4$ & $0$ \\  
    \hline
    $5$ & $\mathbb{Z}/3\mathbb{Z}$ & $3$ & $\mathbb{Z}/3\mathbb{Z}$ \\     
          \hline
          $6$ & $\mathbb{Z}/3\mathbb{Z}\times \mathbb{Z}/3\mathbb{Z}$ & $3$ & $0$ \\   
          \hline
          $7$ & $\mathbb{Z}/6\mathbb{Z}$ & $6$ & $0$ \\   
          \hline
  \end{tabular}
  \caption{Table 1}\label{Table1}
\end{table}
\endgroup

We next recall the definition of the fixed point free action only for bielliptic surfaces of Type 1, 3 and 5. In fact, in the proof of the main results we will only need the explicit description for surfaces of Type 1 and 5. In what follows for a fixed $k$-rational point $P_0\in E_1(k)$ we will denote by $\tau_{P_0}:E_1\to E_1$ the translation by $P_0$. 

\textbf{Type 1 Surfaces.} $G=\mathbb{Z}/2\mathbb{Z}=\langle\sigma\rangle$. Let $P_0\in E_1(k)$ be a nonzero $2$-torsion point. Then $\sigma$ acts on $X$ by 
\[\sigma(P,Q)=(\tau_{P_0}(P), -Q).\] 

\textbf{Type 3 Surfaces.} $G=\mathbb{Z}/4\mathbb{Z}=\langle\phi\rangle$. Assume that $k$ contains a primitive fourth root of unity $i$ and that $E_2$ has complex multiplication by $\mathbb{Z}[i]$. Let $i:E_2\to E_2$ be the extra endomorphism. Let $P_0\in E_1(k)$ be a point of order $4$. Then $\phi$ acts on $X$ by
\[\phi(P,Q)=(\tau_{P_0}(P),i(Q)).\]

\textbf{Type 5 Surfaces.} $G=\mathbb{Z}/3\mathbb{Z}=\langle\rho\rangle$. Suppose that the elliptic curve $E_2$ has complex multiplication by $\mathbb{Z}[\omega]$, where $\omega$ is a primitive third root of unity in $k$, and let $\omega: E_2\to E_2$ be the extra endomorphism. Let $P_0\in E_1(k)$ be a nonzero $3$-torsion point. Then $\rho$ acts on $X$ by
\[\rho(P,Q)=(\tau_{P_0}(P), \omega(Q)).\]

In all cases we will denote by $\pi:X\to S$ the \'{e}tale quotient map, which is of degree $n=|G|$. 
When $G$ is not cyclic, or it is cyclic but its order is not a prime number, then $S$ admits a cyclic cover $\pi':\tilde{S}\to S$ from another bielliptic surface $\tilde{S}$ whose group is cyclic of prime order. We recall these results from \cite{Nue25}. 

\begin{lem}\label{lem1} (cf.~\cite[Lemma 2.4]{Nue25}) Suppose that $\text{ord}(K_S)$ is composite with proper divisor $d$. Then there is a bielliptic surface $\tilde{S}$ sitting as an intermediate \'{e}tale cover  between $S$ and $X$, $\pi:X\xrightarrow{\tilde{\pi}}\tilde{S}\xrightarrow{\pi'}S$, such that $\text{ord}(\omega_{\tilde{S}})=\text{ord}(K_S)/d$. 
\end{lem}
Applying this to surfaces of Type $3$ and $7$ and for $d=2, 3$ respectively,  we get \'{e}tale covers $\tilde{S}\xrightarrow{\pi'}S$ of degrees $2$ and $3$ respectively such that in both cases $\tilde{S}$ is a bielliptic surface of Type 1. 

\begin{lem}\label{lem2} (cf.~\cite[Lemma 2.5]{Nue25}) Suppose that $S=(E_1\times E_2)/G$, with $G\simeq\mathbb{Z}/m\mathbb{Z}\times\mathbb{Z}/\lambda_S\mathbb{Z}$, where $\lambda_S=|G|/\text{ord}(K_S)$. Suppose that $\lambda_S>1$. Then there is a bielliptic surface $\tilde{S}$ sitting as an intermediate \'{e}tale cover  between $S$ and $X$, $\pi:X\xrightarrow{\tilde{\pi}}\tilde{S}\xrightarrow{\pi'}S$, such that $\lambda_{\tilde{S}}=1$ and $\text{ord}(\omega_{\tilde{S}})=\text{ord}(K_S)$. 
\end{lem}

Applying this to surfaces of Type $2, 4$ and $6$, we get \'{e}tale covers $\tilde{S}\xrightarrow{\pi'}S$ such that the bielliptic surface $
\tilde{S}$ is of Type $1, 3$ and $5$ respectively and $\pi'$ is of degree $2, 2$ and $3$ respectively. 

The following Corollary summarizes the results we will need in the proof of \autoref{mainthmintro}. 

\begin{cor}\label{cor1} Let $S=(E_1\times E_2)/G$ be a bielliptic surface over $k$ such that $|G|$ is a composite integer. Then the quotient map $X\xrightarrow{\pi}S$ factors through a finite \'{e}tale cover $\tilde{S}\xrightarrow{\pi'} S$, where $\tilde{S}$ is a Type $1$ surface if $2||G|$, and $\tilde{S}$ is a type $5$ surface otherwise. 
\end{cor}

 
\subsection{Zero-cycles on a Product of Elliptic Curves}\label{sec:bilinear} We also need some facts about the Albanese kernel $T(X)$, where $X=E_1\times E_2$. We will denote by $O_i\in E_i(k)$ the zero element of $E_i$. The subgroup $A_0(X)$ of zero-cycles of degree $0$ is generated by elements of the form $[P,Q]-\deg(P,Q)[O_1,O_2]$. Moreover, every rational point $(P,Q)\in X(k)$ induces a special zero-cycle in the kernel $T(X)$ of the Albanese map,
\[z_{P,Q}:=[P,Q]-[P,O_2]-[O_1,Q]+[O_1, O_2]=\pi_1^\star([P]-[O_1])\cdot\pi_2^\star([Q]-[O_2]).\]
 Here we denoted by $\pi_i: E_1\times E_2\to E_i$ the two projections and by $\cdot$ the intersection product. 
 Next, suppose that $L/k$ is a finite extension and we have a point $(P, Q)\in X(L)$. Then this induces closed points $P_L\in (E_1)_L, Q_L\in (E_2)_L$ and we get a zero-cycle
 $z_{P,Q}=\rho_{L/k\star}(z_{P_L, Q_L})$, where $\rho_{L/k}:X_L\to X$ is the projection. It is classically known that this construction gives a surjective group homomorphism
 \[\varepsilon: \bigoplus_{L/k}(E_1(L)\otimes_\mathbb{Z} E_2(L))\twoheadrightarrow T(X).\] In particular, the zero-cycle $z_{P,Q}$ is \textbf{bilinear} on $P,Q$. 
 
\begin{remark}  We note that Raskind and Spiess (\cite{RS00}) computed the kernel of the map $\varepsilon$. They showed an isomorphism $T(X)\simeq K(k;E_1, E_2)$ with the Somekawa $K$-group attached to $E_1, E_2$. In what follows we won't need this fact. We will only use the bilinearity of $z_{P,Q}$.  
\end{remark}

\section{Main results} 
\subsection{Structure of the Albanese kernel over non-algebraically closed fields} 
We are now ready to prove \autoref{mainthmintro}, which we restate here. 
\begin{theo}\label{mainthm} Let $X=E_1\times E_2$ be a product of elliptic curves over $k$ and $S=(E_1\times E_2)/G$ be a bielliptic surface. Then the Albanese kernel $T(S)$ is a torsion group of exponent $2^2\cdot|G|$ if $2$ divides $|G|$, and $3^2\cdot|G|$ otherwise. 
\end{theo}

\begin{proof} 
The first step will be to reduce to the case when $S$ is either a Type 1 or a Type 5 surface. Suppose that $|G|$ is composite. Then it follows by  \autoref{cor1} that the quotient map $X\xrightarrow{\pi}S$ factors through a finite \'{e}tale map $\tilde{S}\xrightarrow{\pi'}S$, where $\tilde{S}$ is a surface of Type $1$, if $2||G|$, and of Type $5$ otherwise. Set $\varepsilon(\tilde{S})=2$, if $\tilde{S}$ is of Type $1$, and $\varepsilon(\tilde{S})=3$, if $\tilde{S}$ is of Type $5$. 
Note that $\pi'$ is a finite morphism of degree $m=|G|/\varepsilon(\tilde{S})$. It induces a push-forward $\pi'_\star:\CH_0(\tilde{S})\to\CH_0(S)$, which restricts to a homomorphism $A_0(\tilde{S})\xrightarrow{\pi'_\star} A_0(S)$. Moreover, it follows by the universal property of the Albanese variety that $\pi'$ induces a homorphism of abelian varieties $\Alb_{\tilde{S}}\to\Alb_S$, and hence the push-forward map $\pi'_\star$ restricts further to a homomorphism 
$\pi'_\star:T(\tilde{S})\to T(S)$. It follows easily by the identity $\pi'_\star\pi'^\star=m$ that \[m\cdot\frac{T(S)}{\pi'_\star(T(\tilde{S}))}=0\] (see for example the proof of \cite[Theorem 5.2]{GR25}).  Thus, it suffices to show that $T(\tilde{S})$ is torsion of exponent $\varepsilon(\tilde{S})^3$. 

From now on we assume that $S=\tilde{S}$ is a bielliptic surface of Type $1$ or $5$. 
Similarly to the previous step, the projection map $\pi:X\to S$ induces a push-forward
$\pi_\star:T(X)\to T(S)$, such that $T(S)/\pi_\star(T(X))$ is $\varepsilon(S)$-torsion. This reduces the problem to showing that every zero-cycle $z\in\pi_\star(T(X))$ is $\varepsilon(S)^2$-torsion. Recall from \autoref{sec:bilinear} that the group $T(X)$ is generated by zero-cycles of the form $\rho_{L/k\star}(z_{P_L,Q_L})$, where $L/k$ ranges over all finite extensions of $k$ and $(P, Q)\in X(L)$. We have a commutative diagram
\[\begin{tikzcd} X_L\arrow[r, "\pi_L"] \ar[d, "\rho_{L/k}"] & S_L  \arrow[d, "\rho_{L/k}"] \\
	     X\arrow[r, "\pi"] & S ,
\end{tikzcd} 
\]  which yields $\rho_{L/k\star}\pi_{L\star}=\pi_\star\rho_{L/k\star}$. 
Thus, it is enough to show that $\pi_{L\star}(z_{P_L,Q_L})$ is $\varepsilon(S)^2$-torsion for $(P_L,Q_L)\in X_L(L)$ where $L/k$ ranges over all finite extensions of $k$. As the computation is uniform on $L/k$, for simplicity we assume that $(P,Q)\in X(k)$. 
The method is similar for both types, but we treat each case separately for clarity. 

\textbf{Type 1:} Suppose $G=\mathbb{Z}/2\mathbb{Z}=\langle\sigma\rangle$ and the action is given on rational points by $\sigma(P,Q)=(P+P_0, -Q)$, where $P_0\in E_1[2](k)$, $P_0\neq O_1$.  
We have:
\begin{eqnarray}
\pi_\star(z_{P,Q})= && \pi_\star([P,Q]-[P,O_2]-[O_1,Q]+[O_1,O_2])\nonumber\\=
&&\pi_\star([P+P_0,-Q]-[P+P_0, O_2]-[O_1, Q]+[O_1, O_2]).\\
\pi_\star(z_{P+P_0, -Q})=&&\pi_\star([P+P_0,-Q]-[P+P_0, O_2]-[O_1, -Q]+[O_1, O_2]).
\end{eqnarray}
By subtracting these two relations we get 
\[\pi_\star(z_{P,Q}-z_{P+P_0,-Q})=\pi_\star([O_1, -Q]-[O_1,Q])=\pi_\star([P_0, Q]-[O_1, Q]).\] 
Since $P_0$ is a $2$-torsion point of $E_1$, on $\CH_0(E_1)$ we have a rational equivalence $2([P_0]-[O_1])=0$, which implies that the zero-cycle $t_1:=[P_0, Q]-[O_1, Q]\in T(X)$ is $2$-torsion. We conclude that 
\begin{equation}\label{eq1}
\pi_\star(z_{P,Q})-\pi_\star(z_{P+P_0,-Q})=\pi_\star(t_1)\text{ is }2-\text{torsion}. 
\end{equation}
On the other hand, bilinearity gives,
\[
\pi_\star(z_{P+P_0,-Q})=-\pi_\star(z_{P,Q})+\pi_\star(z_{P_0, -Q}).\]
Since $2P_0=O_1$, bilinearity again yields that the zero-cycle $t_2:=z_{P_0, -Q}$ is also $2$-torsion. We conclude that 

\begin{equation}\label{eq2}
\pi_\star(z_{P+P_0,-Q})+\pi_\star(z_{P,Q})=\pi_\star(t_2)\text{ is }2-\text{torsion}.
\end{equation}
Combining equations \eqref{eq1}, \eqref{eq2} we conclude that $2\pi_\star(z_{P,Q})$ is $2$-torsion, and hence $\pi_\star(z_{P,Q})$ is $4$-torsion. 

\textbf{Type 5:} Suppose that $G=\mathbb{Z}/3\mathbb{Z}=\langle\rho\rangle$ and the action is given by $\rho(P,Q)=(P+P_0, \omega(Q))$, where $P_0\in E_1[3](k)$, $P_0\neq O_1$, and $\omega$ is a primitive third root of unity in $k$ inducing an extra endomorphism $\omega:E_2\to E_2$. The proof in this case is very analogous. 
First, recall that $\omega^2+\omega+1=0$. Using the bilinearity of $z_{P,Q}$, this yields,
\begin{eqnarray*}
z_{P,Q}+z_{P+P_0,\omega(Q)}+z_{P+2P_0, \omega^2(Q)}=s_1,
\end{eqnarray*} where $s_1:=z_{P_0,\omega(Q)}+z_{2P_0,\omega^2(Q)}$ is a $3$-torsion element. Thus, we conclude that 
\begin{equation}
\pi_\star(z_{P,Q}+z_{P+P_0,\omega(Q)}+z_{P+2P_0, \omega^2(Q)})=\pi_\star(s_1)\text{ is }3-\text{torsion}. 
\end{equation}

Moreover, using the group action we get: 
\begin{eqnarray*}
\pi_\star(z_{P,Q})= && \pi_\star([P,Q]-[P,O_2]-[O_1,Q]+[O_1,O_2])\\=
&&\pi_\star([P+P_0,\omega(Q)]-[P+P_0, O_2]-[O_1, Q]+[O_1, O_2]).\\
\pi_\star(z_{P+P_0, \omega(Q)})=&&\pi_\star([P+P_0,\omega(Q)]-[P+P_0, O_2]-[O_1, \omega(Q)]+[O_1, O_2]).
\end{eqnarray*}
Subtracting the two relations yields,
\[\pi_\star(z_{P,Q}-z_{P+P_0,\omega(Q)})=\pi_\star([O_1,\omega(Q)]-[O_1,Q])=\pi_\star([O_1,\omega(Q)]-[P_0,\omega(Q)]).\] Since $P_0$ is a $3$-torsion point of $E_1$, on $\CH_0(E_1)$ we have a rational equivalence $3([P_1]-[O_1])=0$, and hence the zero-cycle $s_2:=[O_1,\omega(Q)]-[P_0,\omega(Q)]\in \CH_0(X)$ is $3$-torsion. Thus, the last relation yields that 
\[\pi_\star(z_{P,Q}-z_{P+P_0,\omega(Q)})=\pi_\star(s_2)\text{ is }3-\text{torsion}.\] Replacing $P$ with $P+P_0$ and $Q$ with $\omega(Q)$ gives 
\[\pi_\star(z_{P+P_0,\omega(Q)}-z_{P+2P_0,\omega^2(Q)})\text{ is }3-\text{torsion}.\] Combining these with equation \ref{eq2} yields that $3\pi_\star(z_{P,Q})$ is $3$-torsion, and hence $\pi_\star(z_{P,Q})$ is $9$-torsion as required. 

$\hfill\Box$
\end{proof}

\subsection{Results over $p$-adic fields}
In this section we prove \autoref{thm2intro}. Throughout this section we will be working over a finite extension $k$ of $\mathbb{Q}_p$ with ring of integers $\mathcal{O}_k$ and residue field $\mathbb{F}$. We keep the same notation from the previous sections and also assume that $p\geq 5$. 
%
%
%
%
%
%
%
%

We will  construct an example of a product $X=E_1\times E_2$ of elliptic curves with split multiplicative reduction over a $p$-adic field $k$ and a bielliptic surface $S$ of Type 1 such that the subgroup $\pi_\star(T(X))$ of $T(S)$ is nontrivial. To do this, we will use the Brauer-Manin pairing. We start with some preliminary results on the transcendental Brauer group of such a Type 1 surface.

\begin{lem}\label{Brcomputation} Let $X=E_1\times E_2$ be a product of elliptic curves defined over $\mathbb{Q}$.  Suppose that $E_1, E_2$ are non-isogenous over $\overline{\mathbb{Q}}$. Let $S=X/G$, $G=\mathbb{Z}/2\mathbb{Z}$, be a bielliptic surface of Type 1. Then the pullback map 
\[\pi^\star: \Br(S_{\overline{\mathbb{Q}}})\to \Br(X_{\overline{\mathbb{Q}}})\] is injective and its image is precisely the $G$-invariant subgroup of $\Br(X_{\overline{\mathbb{Q}}})[2]$. The same result holds true with $\overline{\mathbb{Q}}$ replaced by $\overline{\mathbb{Q}}_p$ for every prime $p$. 
\end{lem} 
\begin{proof} Considering the base change to $\mathbb{C}$, it follows by \cite[Theorem B]{FTVB22} that the pullback 
$\pi^\star: \Br(S_{\mathbb{C}})\to \Br(X_{\mathbb{C}})$ is injective, because $E_1, E_2$ are non-isogenous. It then follows by \cite[Proposition 5.2.3]{CTS21} that we have isomorphisms
\[\Br(S_{\mathbb{C}})\simeq \Br(S_{\overline{\mathbb{Q}}})\simeq \Br(S_{\overline{\mathbb{Q}}_p})\] for every prime $p$ and similar isomorphisms are true for $X$. 

Moreover, since $S$ is of Type 1, we have isomorphisms \[\Br(S_{\mathbb{C}})\simeq H^2(S_{\mathbb{C}},\mathbb{Z})_{\tor}\simeq (\mathbb{Z}/2\mathbb{Z})^2\] (see Table \ref{Table1}). Putting all this together, we see that the image of the pullback map $\Br(S_{\overline{\mathbb{Q}}})\xrightarrow{\pi^\star} \Br(X_{\overline{\mathbb{Q}}})$ is a subgroup of $\Br(X_{\overline{\mathbb{Q}}})[2]$ of order $4$. It remains to show that it is precisely the $G$-invariant subgroup of the latter. \cite[Proposition 3.3]{SZ12} gives an isomorphism
\[\Br(X_{\overline{\mathbb{Q}}})[2]\simeq\frac{\Hom(E_1[2], E_2[2])}{(\Hom(E_1, E_2))/2}\simeq \Hom(E_1[2], E_2[2]).\] The last equality follows because $E_1, E_2$ are non-isogenous. Let $G=\langle\sigma\rangle$ act on $X_{\overline{\mathbb{Q}}}$ via $\sigma(P,Q)=(P+P_0, -Q)$. Since the diagram 
\[\begin{tikzcd}
X \ar{r}{\sigma} \ar{dr}{\pi} & X \ar{d}{\pi} \\
& S
\end{tikzcd}\] is commutative, it follows that $\pi^\star=\sigma^\star\pi^\star$, which shows that $\pi^\star(\Br(S_{\overline{\mathbb{Q}}})$ lies in the $G$-invariant subgroup of $\Br(X_{\overline{\mathbb{Q}}})$. Equality will follow if we show that the latter is isomorphic to $(\mathbb{Z}/2\mathbb{Z})^2$. But this is clear, since every homomorphism $f:E_1[2]\to E_2[2]$ sending $P_0$ to the zero element is fixed by the $G$-action. 

$\hfill\Box$

\end{proof}

We next recall some facts about the \textit{Brauer-Manin pairing}. Let $Y$ be a smooth projective geometrically connected surface over $k$. 
 The Brauer-Manin pairing (see \cite[p.~53]{CT95}), 
\[\langle,\rangle:\CH_0(Y)\times\Br(Y)\to \mathbb{Q}/\mathbb{Z},\] is defined as follows. For the class $[P]$ of a closed point and a Brauer class $\alpha\in \Br(Y)$, 
$\langle [P],\alpha\rangle:=\rm{Cor}_{k(P)/k}(\iota_P^\star(\alpha))$, where $\iota_P:\rm{Spec}(k(P))\hookrightarrow X$ is the closed immersion, and $\rm{Cor}_{k(P)/k}:\Br(k(P))\to\Br(k)$ is the Corestriction map of Galois cohomology. The pairing induces a homomorphism
\[\varepsilon_2: \CH_0(Y)/2\to \Hom(\Br(Y)[2],\mathbb{Q}/\mathbb{Z}).\]

The following proposition is the analog of \cite[Proposition 3.9]{GL24}. 
\begin{prop}\label{BMsquare} Let $X=E_1\times E_2$ be a product of elliptic curves over $k$. Suppose that $E_1, E_2$ are not isogenous over $\overline{k}$. Let $S=X/G$, $G=\mathbb{Z}/2\mathbb{Z}$, be a bielliptic surface of Type 1. Then the projection map $\pi:X\to S$ induces a commutative square 
\[\begin{tikzcd}
T(X)\ar{r}{\varepsilon_2^X} \ar{d}{\pi_\star} & \Hom(\Br(X)/\Br_1(X))\ar{d}{\widehat{\pi^\star}}\\ 
\pi_\star(T(S))\ar{r}{\varepsilon_2^S} & \Hom(\Br(S)/\Br_1(S)).
\end{tikzcd}\] Here $\Br_1(X))=\ker(\Br(X)\to \Br(X_{\overline{k}}))$ is the \textit{algebraic Brauer group} of $X$. 
\end{prop}
The proof is exactly analogous to the proof of \cite[Proposition 3.9]{GL24}, using first \cite[Lemma 3.7]{GL24}. We just note that the image of the Albanese kernel $T(X)$ under the map $\varepsilon_2^X$ lies indeed in  $\Hom(\Br(X)/\Br_1(X))$. This is because $E_1, E_2$ being non-isogenous over $\overline{k}$ implies that the $\Gal(\overline{k}/k)$-action on the N\'{e}ron-Severi group $\NS(X_{\overline{k}})$ is trivial, and then the result follows from \cite[Prop 2.16]{Gaz24}. 
\medskip

%
We are now ready to prove \autoref{thm2intro}, which we restate here.

\begin{theo} There exists a product $X=E_1\times E_2$ over a $p$-adic field $k$ and a bielliptic surface $S=X/G$ of Type 1 such that the group $T(S)$ has nontrivial $2$-torsion.  
\end{theo} 

\begin{proof} 
Our strategy will be to find a pair of elliptic curves $E_1, E_2$ over $\mathbb{Q}$, a prime $p$ and a bielliptic surface $S=X/G$ of Type 1 (where as usual $X=E_1\times E_2$) such that there exists a zero-cycle $z\in T(X_{\mathbb{Q}_p})$ with the property $\pi_\star(z)\neq 0$. We start with some local analysis over an arbitrary $p$-adic field $k$ with $p\geq 5$, but we implicitly assume that $X$ is defined over a number field, so that \autoref{Brcomputation} can be used. 

Consider a product $X=E_1\times E_2$ of elliptic curves over $k$. Suppose that both $E_1, E_2$ have split multiplicative reduction. Then, as rigid $p$-analytic spaces, $E_i\simeq \mathbb{G}_m/q_i^\mathbb{Z}$ is a Tate elliptic curve for $i=1,2$, where $q_i\in k^\times$. This $p$-adic uniformization implies that for every $n\geq 1$ there is a short exact sequence of $\Gal(\overline{k}/k)$-modules
\[0\to\mu_{n}\to E_{i}[n]\to \mathbb{Z}/n\mathbb{Z}\to 0.\] In fact, if $\zeta_n\in\overline{k}^\times$ is a primitive $n$-th root of unity, then 
\[E_i[n]=\{\zeta_n^a\sqrt[n]{q^b}: 0\leq a, b\leq n-1\}.\]

From now on we assume additionally that $E_1, E_2$ are non-isogenous over $\overline{k}$. 
Consider the Brauer-Manin pairing \[\langle,\rangle_X:\CH_0(X)\times\Br(X)\to\mathbb{Q}/\mathbb{Z}.\] 
Because $E_1, E_2$ both have split multiplicative reduction, it follows by \cite[Theorem 1.2]{Yam05} that the left kernel of the pairing is the maximal divisible subgroup of $\CH_0(X)$. As mentioned in \autoref{BMsquare}, $\langle,\rangle_X$ restricts to a pairing  
\[\langle,\rangle_X:T(X)\times \frac{\Br(X)}{\Br_1(X)}\to \mathbb{Q}/\mathbb{Z}.\] 
Restricting further to the $n$-torsion subgroup of the Brauer group, we obtain

\[\langle,\rangle_X:\frac{T(X)}{n}\times \frac{\Br(X)[n]}{\Br_1(X)[n]}\to \mathbb{Z}/n\mathbb{Z}.\] 
Let $H_1$ be the right kernel of this pairing and $Q$ be the quotient of $\displaystyle\frac{\Br(X)[n]}{\Br_1(X)[n]}$ by $H_1$. 
The nonzero elements of $Q$ are those transcendental Brauer classes that pair nontrivially with elements of $T(X)/n$.  Recall from \cite[Proposition 3.3]{SZ12} that we have an isomorphism 
\[\Br(X)[n]/\Br_1(X)[n]\simeq\Hom_{\Gal(\overline{k}/k)}(E_1[n],E_2[n]).\]
Thus, to determine the quotient $Q$ we need to understand which $\Gal(\overline{k}/k)$-equivariant homomorphisms $f:E_1[n]\to E_2[n]$  pair nontrivially with $T(X)/n$. This is computed in more generality in \cite[Theorem 1.6]{Gaz19}, but in the special case when both elliptic curves have split multiplicative reduction it also follows from \cite[Proof of Theorem 4.3]{Yam05}. If we assume for simplicity that $E_i[n]\subset E_i(k)$ for $i=1,2$, the answer is precisely $\Hom(\mu_n, \mathbb{Z}/n\mathbb{Z})$.  Notice that $\Hom(\mu_n, \mathbb{Z}/n\mathbb{Z})$ is considered a quotient of $\Hom(E_1[n], E_2[n])$ via the compositions: 
\[\mu_n\hookrightarrow E_1[n]\xrightarrow{\phi} E_2[n]\twoheadrightarrow \mathbb{Z}/n\mathbb{Z}, \] where the injection  $\mu_n\hookrightarrow E_1[n]$ and the surjection $E_2[n]\twoheadrightarrow \mathbb{Z}/n\mathbb{Z}$ are the ones induced by the $p$-adic uniformizations. 

Next we focus on the case $n=2$. Assume that $E_i[2]\subset E_i(k)$ for $i=1,2$. We can then choose generators of $E_1[2]=\langle P_1, P_2\rangle$ such that $P_1\in\mu_2$ and $P_2\in E_1[2]/\mu_2$. Consider the bielliptic surface $S=X/G$, with $G=\mathbb{Z}/2\mathbb{Z}=\langle\sigma\rangle$ acting on $E_1$ by translation by $P_2$. It follows by \autoref{Brcomputation} and its proof that the subgroup $H_2=\pi^\star(\Br(S)/\Br_1(S))$ of $\Hom(E_1[2],E_2[2])$ is precisely the one consisting of those homomorphisms $g:E_1[2]\to E_2[2]$ that send $P_2$ to $0$. We see that under the choices we made there exists a homomorphism $f\in H_2$ such that its image in the quotient $Q=\Hom(\mu_n, \mathbb{Z}/n\mathbb{Z})$ is nonzero. This means that there exists $\alpha\in \pi^\star(\Br(S)/\Br_1(S))$ and $z\in T(X)/2$ such that 
\[\langle \pi_\star(z), \alpha\rangle_S=\langle z, \pi^\star(\alpha)\rangle_X\neq 0.\] (Note that the first equality follows by \cite[Lemma 3.7]{GL24}). We conclude that $\pi_\star(z)\in T(S)$ is nontrivial. 

To close the argument, we need to verify that such elliptic curves exist, and we can even choose them to be defined over $\mathbb{Q}$, so that we can apply \autoref{Brcomputation}. These are easy to find. For example consider the elliptic curves 
\[E_1: y^2=x^3-14931x+220590,\;\;\;\;\;E_2:y^2=x^3-3171x+68510\] with LMFDB labels $33.a2$ and $198.a2$ respectively. Both curves have split multiplicative reduction at $p=11$. At the same time, $E_2$ has non-split multiplicative reduction at $p=2$, while $E_1$ has good reduction at $2$, and hence $E_1, E_2$ are not isogenous over $\overline{k}$.

$\hfill\Box$
\end{proof}

 We note that in the above proof the choice to consider elliptic curves with bad reduction was necessary. If both elliptic curves have instead good reduction over $k$, then given that $p>2$, it follows by \cite[Corollary 3.5.1]{RS00} that the group $T(X)$ is $2$-divisible. Thus, there is no $z\in T(X)$ such that $\pi_\star(z)\neq 0$. This leads to the following Corollary, which improves \autoref{mainthm} in this case.
 
 \begin{cor} Let $X=E_1\times E_2$ be a product of elliptic curves over a $p$-adic field $k$. Assume that $p\geq 5$ and both elliptic curves have good reduction. Let $S=(E_1\times E_2)/G$ be a bielliptic surface of Type 1 (resp. of Type 5) over $k$. Then the Albanese kernel $T(S)$ is $2$-torsion (resp. $3$-torsion). 
 \end{cor}

\begin{remark}\label{finitefield}  It is easy to see that if $E_1, E_2$ have good reduction, then so does the bielliptic surface $S$. Let $\overline{S}$ be the reduction surface over the residue field $\mathbb{F}$. It follows by the unramified class field theory of Kato and Saito (\cite[Proposition 9]{KS83}) that the group $A_0(\overline{S})$ fits into a short exact sequence
\[0\to \NS(\overline{S})_{\tor}\to A_0(\overline{S})\to \Alb_{\overline{S}}(\mathbb{F})\to 0,\] where $\NS(\overline{S})$ is the N\'{e}ron-Severi group of $\overline{S}$. By passing to a finite unramified extension if necessary, we may assume that $\NS(\overline{S})_{\tor}$ coincides with the torsion subgroup of the geometric N\'{e}ron-Severi group $\NS(\overline{S}_{\overline{\mathbb{F}}})$, which is nontrivial, since there is a class corresponding to the abelian \'{e}tale cover $\overline{X}\xrightarrow{\overline{\pi}}\overline{S}$. We conclude that the Albanese kernel of the reduction $T(\overline{S})$ is nontrivial. Given that the specialization map induces a surjective group homomorphism 
\[A_0(S)\to A_0(\overline{S})\] (cf.~\cite[Th\'{e}or\`{e}me 0.1]{CT11}), it is likely that even in the good reduction case $T(S)\neq 0$, but this is not something we can verify in the current article. 
\end{remark}


\begin{thebibliography}{6}
%

\bibitem{BDF10} G. Bagnera and M. De Franchis, {\sl Le nombre $\rho$ de Picard pour les surfaces hyperelliptic.\/}, Rend. Circ. Mat. Palermo,  \textbf{10} (1910). 

\bibitem{BKL76} S. Bloch, A. Kas and D. Lieberman, {\sl Zero-cycles on surfaces with $p_g=0$.\/}, Compositio Math.,  \textbf{33}, no.~2 (1976), 135-145. 


  \bibitem{CTR91} J.-L. Colliot-Th\'{e}l\`{e}ne and Wayne Raskind, \textit{Groupe de Chow de codimension deux des vari\'{e}t\'{e}s d\'{e}finies sur un corps de nombres: un th\'{e}or\`{e}me de finitude pour la torsion}, Inventiones Math. \textbf{105} (1991), no.~2, pp.~221-245. \ 
   
  
   \bibitem{CT95} J.-L. Colliot-Th\'{e}l\`{e}ne, \textit{L'arithm\'{e}tique du groupe de Chow de z\'{e}ro cycles}, J. Th\'{e}or. Nombres Bordeaux \textbf{7} (1995), no.~1, p.~ 51–73. \ 
  
  \bibitem{CT11} J.-L. Colliot-Th\'{e}l\`{e}ne, \textit{Groupe de Chow de z\'{e}ro cycles sur des vari\'{e}t\'{e}s $p$-adiques (d'apr\`{e}s S. Saito, K. Sato et al.)}, Ast\'{e}risque \textbf{339} (2011), Exp. No. 1012, vii, 1–30. \ 
  
  \bibitem{CTS21} J.-L. Colliot-Th\'{e}l\`{e}ne and A. Skorobogatov, \textit{The {B}rauer-{G}rothendieck group}, Ergebnisse der Mathematik und ihrer Grenzgebiete. 3. Folge. A Series of Modern Surveys in Mathematics, vol.~\textbf{71}, Springer, Cham. 2021.  
  
 \bibitem{FTVB22} E. Ferrari, S. Tirabassi and M. Vodrup. With an Appendix by J. {B}ergstr\"om and S. Tirabassi, \textit{On the {B}rauer group of bielliptic surfaces}, Doc. Math. \textbf{27} (2022), pp.~383--425. \ 
 
\bibitem{Gaz19} E. Gazaki, \textit{A Tate duality theorem for local Galois symbols II; the semi-abelian case}, J. of Number Theory \textbf{204} (2019), p.~532-560. 
 
 \bibitem{Gaz24} E. Gazaki, with an Appendix by A. Koutsianas, \textit{Weak approximation for 0-cycles on a product of elliptic curves}, Math Annalen \textbf{388} (2024), no.~2, p.~1539--1568. 
 
 
\bibitem{GL24} E. Gazaki and J. Love, \textit{Local and local-to-global principles for zero-cycles on geometrically Kummer K3 surfaces.}  Joint with Jonathan Love. Submitted for Publication. Arxiv: https://arxiv.org/abs/2402.12588 

\bibitem{GR25} E. Gazaki and J. Rathore, \textit{Zero-cycles on quasi-projective Surfaces over p-adic Fields}. Submitted for Publication. Arxiv:2507.20076

\bibitem{KS83} K. Kato and S. Saito, \textit{Unramified class field theory of arithmetic surfaces}. Annals Math. \textbf{118}, 241-275 (1983). 
 
 \bibitem{Nue25} H. Nuer, \textit{Stable sheaves on bielliptic surfaces: From the classical to the modern}. Math. Z. \textbf{309} (2025) no.~3, Paper no.~39, 66pp. 
 
 \bibitem{RS00} W. Raskind and M. Spiess,  {\sl Milnor K-Groups and Zero-Cycles on Products of
			Curves over $p$-Adic Fields}, Compositio. Mathematica, \textbf{121} (2000), 1-33. \
    
		
\bibitem{SS10} S. Saito and K. Sato, {\sl A finiteness theorem for zero-cycles over p-adic fields.\/},  Ann. of Math., Volume {\bf 172} (2010), 1593-1639. \


\bibitem{Sil85} J. H. Silverman, {\sl The arithmetic of elliptic curves\/}, Grad. Texts in Math. \textbf{106} (1985), Springer. \

\bibitem{SZ12} A. Skorobogatov and Y. Zarhin, \textit{The Brauer group of Kummer surfaces and torsion of elliptic curves}, J. Reine Angew. Math. \textbf{666} (2012), p.~115-140. 

\bibitem{Suw69} T. Suwa, {\sl On hyperelliptic surfaces\/}, J. Fac. Sci. Univ. Tokyo Sect. I,\textbf{16}, 468-476 (1970), 1969. \


\bibitem{Yam05} T. Yamazaki, {\sl On {C}how and {B}rauer groups of a product of {M}umford curves.\/}, Math. Annalen, Volume {\bf 333}, no.~3, (2005), 549--567. \


\end{thebibliography}
\end{document}